\documentclass[11pt]{amsart}
\usepackage[arrow,matrix,graph,frame,poly,arc,tips]{xy}
\usepackage{amscd,amssymb,subfigure,epsfig,color}
\usepackage[width=5.65in,height=8.0in,centering]{geometry}
\usepackage{amsmath}
\usepackage{fancyvrb,verbatim}
\usepackage[colorlinks,plainpages,backref]{hyperref}

\theoremstyle{plain}
\newtheorem{thm}{Theorem}[section]
\newtheorem{lem}[thm]{Lemma}
\newtheorem{prp}[thm]{Proposition}
\newtheorem{cor}[thm]{Corollary}
\theoremstyle{definition}
\newtheorem{rmk}[thm]{Remark}

\newtheorem{exa}[thm]{Example}

\numberwithin{equation}{section}

\newcommand{\xymat}{\SelectTips{cm}{}\xymatrix}

\newcommand{\abs}[1]{{\lvert #1 \rvert}}

\newcommand{\set}[1]{\left\{#1\right\}}
\newcommand{\into}{\hookrightarrow}
\newcommand{\onto}{\twoheadrightarrow}

\newcommand{\A}{{\ensuremath{\mathcal A}}}

\newcommand{\F}{{\mathcal F}}
\newcommand{\ba}{{\mathbf a}}
\newcommand{\X}{{\mathfrak X}}
\newcommand{\kk}{{C}}
 
\newcommand{\Z}{{\mathbb Z}}
\newcommand{\Q}{{\mathbb Q}}

\newcommand{\C}{{\mathbb C}}

\renewcommand{\O}{{\mathcal O}}

\newcommand{\chihom}{{\chi^+}}
\newcommand{\OSigma}{{\overline{\Sigma}}}
\newcommand{\oma}{{\omega_\ba}}

\renewcommand{\P}{{\mathbb P}}

\DeclareMathOperator{\rank}{rank}

\DeclareMathOperator{\codim}{codim}

\DeclareMathOperator{\Der}{Der}

\DeclareMathOperator{\ch}{ch}
\DeclareMathOperator{\td}{td}

\DeclareMathOperator{\id}{id}

\begin{document}
%%%%%%%%%%%%%%%%%%%%%%%%%%%%%%%%%%%%%%%%%%%%%%%%%%%%%%%%%%%%%%%%%%%%%%%%%%%%%%%

\title{A geometric deletion-restriction formula}

\author{Graham Denham}
\address{
G.~Denham\\
Department of Mathematics\\
University of Western Ontario\\
London, ON N6A 5B7\\
Canada
}
\email{gdenham@uwo.ca}
\thanks{GD was partially supported by a grant from NSERC of Canada.}

\author{Mehdi Garrousian}
\address{
M.~Garrousian\\
Department of Mathematics\\
University of Western Ontario\\
London, ON N6A 5B7\\
Canada
}
\email{mgarrou@uwo.ca}

\author{Mathias Schulze}
\address{
M.~Schulze\\
Department of Mathematics\\
Oklahoma State University\\
Stillwater, OK 74078\\
United States}
\email{mschulze@math.okstate.edu}
\thanks{}

\date{\today}

%%%%%%%%%%%%%%%%%%%%%%%%%%%%%%%%%%%%%%%%%%%%%%%%%%%%%%%%%%%%%%%%%%%%%%%%%%%%%%%

\begin{abstract}
In this paper, we recover the characteristic polynomial of an arrangement of hyperplanes by computing the rational equivalence class of the variety defined by the logarithmic ideal of the arrangement. The logarithmic ideal was introduced in \cite{CDFV09} in a study of the critical points of the master function. The above result is used to understand the asymptotic behavior the Hilbert series of the logarithmic ideal. As an application, we prove the Solomon-Terao formula under the tame hypothesis by identifying each side of the formula with a certain specialization of the Hilbert series of the logarithmic ideal.
\end{abstract}

% 52C35 Discrete geometry - Arrangements of points, flats, hyperplanes
% 16W25 Rings and algebras with additional structure - Derivations, actions of Lie algebras


\maketitle
\tableofcontents

%%%%%%%%%%%%%%%%%%%%%%%%%%%%%%%%%%%%%%%%%%%%%%%%%%%%%%%%%%%%%%%%%%%%%%%%%%%%%%%
\section{Introduction}
%%%%%%%%%%%%%%%%%%%%%%%%%%%%%%%%%%%%%%%%%%%%%%%%%%%%%%%%%%%%%%%%%%%%%%%%%%%%%%%

The characteristic polynomial is a ubiquitous combinatorial invariant
of constructions associated with hyperplane arrangements.  For example, its coefficients are Whitney numbers of a lattice, as well as Betti numbers of the complement of a complex arrangement~\cite{OS80}.  Over a finite field, the polynomial counts rational points in an arrangement complement: see, e.g. \cite{At96}.  For real arrangements, the coefficients are average projection volumes: see \cite{KS11}.  For
certain complex arrangements (locally free), the coefficients are Chern numbers of the sheaf of logarithmic $1$-forms: \cite{MuSc01}.
For any arrangement, its Chern-Schwartz-MacPherson class is given by its characteristic polynomial: \cite{Al11}.
The main result of this paper adds yet another example to the (incomplete) list above of formulas for the characteristic polynomial of a hyperplane arrangement, this time via the rational equivalence 
class of a biprojective variety associated with the following problem.

Fix a central rank-$\ell$ arrangement $\A$ of $n$ hyperplanes in a $m$-dimensional complex vector space $V$ defined by $n$ linear functions $f_1,\ldots,f_n\in V^*$ and let $f=f_1\cdots f_n$. 
Denote the hyperplanes by $H_i=\ker(f_i)$ for $1\leq i\leq n$.
Given a vector of weights $\lambda=(\lambda_1,\ldots,\lambda_n)\in\C^n$, we consider
the {\em master function}
\begin{equation}\label{eq:defPhi}
\Phi_\lambda=\prod_{i=1}^n f_i^{\lambda_i}.
\end{equation}
This multi-valued function has zeros and poles on the variety $\bigcup_{i=1}^n H_i$ defined by $\A$; accordingly let $M=V\setminus\bigcup_{i=1}^n H_i$.  

We denote the set of critical points of $\Phi_\lambda$ on $M$ by
\[
\Sigma_\lambda=\set{x\in M\mid d\Phi_\lambda(x)=0}.
\]
For suitable arrangements $\A$ and choices of weight $\lambda$, the
the critical points of the master function index a basis of solutions
to a physically significant PDE: see, for example, \cite{RV95,SV03,Va03}. 
The core of the theory depends on properties of hyperplane arrangements, as Varchenko shows in \cite{Va11}. 
This has been the primary motivation for a study of the critical points of master functions in \cite{Va95,OT95,Silv96,CDFV09}. 
The question of finding extremal values of \eqref{eq:defPhi} in the case of real defining equations $\set{f_i}$ is closely related, and complementary results have been obtained in \cite{CHKS06,HKS05}, motivated by applications in algebraic statistics.

Let 
\begin{equation}\label{eq:omega_a}
\omega_\ba=\sum_{i=1}^n a_i\frac{df_i}{f_i},
\end{equation}
where $a_1,\ldots,a_n$ are coordinate functions on $\C^{n}$ as the space of weights. Let $\omega_\lambda$ denote the specialization of $\omega_\ba$ with $a_i=\lambda_i$ for $1\leq i\leq n$: then $\omega_\lambda$ is the
logarithmic derivative of $\Phi_\lambda$, and we see
$\Sigma_\lambda=\set{x\in M\mid\omega_\lambda(x)=0}$.
In order to consider critical sets of \eqref{eq:defPhi} for a fixed arrangement and all $\lambda\in\C^{n}$, let $\Sigma(\A)$ be the subvariety of
$M\times \C^{n}$ given by the vanishing of $\omega_\ba$, and let $\OSigma(\A)$ be its closure in $V\times\C^{n}$.  The diagonal actions of $\C^*$ on $V$ and $\C^{n}$ preserve $\OSigma(\A)$, so we
let 
\[
\X(\A)=\OSigma(\A)/(\C^*\times\C^*),
\]
a subvariety of $\P V\times \P^{n-1}$.
These varieties were introduced in \cite{CDFV09} and studied further in \cite{DS10}: in particular, $\X(\A)$ is irreducible, has codimension $\ell=\rank\A$, and is smooth over $\P M\times\P^{n-1}$.

The variety $\X(\A)$
%$\OSigma(\A)$ 
can be described using the module $D(\A)$ of logarithmic derivations along $\A$:
by \cite[Thm.~2.9]{CDFV09}, it is the (biprojective) zero-locus of the ideal $I(\A)=\langle D(\A),\oma\rangle$ obtained by contraction of $\oma$ along logarithmic vector fields.  (Details appear in \S\ref{ss:critical}).  An arrangement $\A$ is said to be free
if $D(\A)$ is a free module.
By \cite[Thm.~2.13]{CDFV09}, the ideal $I(\A)$ is a complete intersection if and only $\A$ is free.  
Moreover, in this case, $I(\A)$ is generated in bidegrees $\set{(d_i,1)\mid1\leq i\leq \ell}$, where the numbers $\set{d_i+1}$, indexed in nondecreasing order, are called the exponents of $\A$.  (We assume for the moment that $m=\ell$.)

%Let $\set{d_i}$ be the degrees of any minimal set of homogeneous %generators of $D(\A)$, ordered non-decreasingly and using the grading %convention from \S\ref{sec:log}.
%Then $\set{d_i+1}$ are called the exponents of $\A$, and $\A$ is %called free if $D(\A)$ is a free module (neccessarily of rank $m$).
%By \cite[Thm.~2.13]{CDFV09}, the ideal $I(\A)$ is a complete %intersection if and only $\A$ is free. 
%In this case, $I(\A)$ is generated in bidegrees %$\set{(d_i,1)\mid1+m-\ell\leq i\leq m}$.

Now consider the Chow ring
\[
A^\bullet=A^{\bullet}(\P V\times \P^{n-1})\cong\Z[h,k]/(h^\ell,k^n)
\]
where $h=[H]$ and $k=[K]$ denote the classes of hyperplanes $H$, $K$ in $\P V$ and $\P^{n-1}$ respectively (see, e.g., \cite[Ex.~8.4.2]{FultonBook}).
If $\A$ is free, from the above-mentioned degrees of the generators of $I(\A)$, we compute 
\begin{equation}\label{eq:Xfac}
[\X(\A)]=\prod_{i=1}^\ell\big(d_i h+k\big)\in A^\bullet.
\end{equation}

There is an analogous result for the characteristic polynomial $\chi(\A,t)$ of $\A$.
% (defined as in \eqref{eq:chpoly} instead of %\cite[Def.~2.52]{OTbook}):
Terao's Factorization Theorem~\cite{Te81} states that, if $\A$ is free, then the characteristic polynomial factors as
\begin{equation}\label{eq:chifac}
\chi(\A,t)=\prod_{i=1}^\ell(t-d_i-1)\in\Z[t].
\end{equation}

By comparison of \eqref{eq:Xfac} and \eqref{eq:chifac}, we deduce that
\[
[\X(\A)]=\chihom(-h,k-h),
\]
where $\chihom(\A,s,t)=s^\ell\chi(\A,t/s)$ denotes the homogenized characteristic polynomial.
Our main result is that this formula holds in general.

\begin{thm}\label{thm:intcycle}
For any central arrangement \A, we have 
\[
[\X(\A)]=\chihom(\A,-h,k-h)\in A^\bullet
\]
That is, the cycle of the variety $\X(\A)$ determines the characteristic polynomial of $\A$.
\end{thm}

On the other hand, we note that the variety
$\X(\A)$ itself is not a purely combinatorial object: there exist
arrangements $\A_1$, $\A_2$ with the same characteristic polynomial (indeed, the same underlying matroid) for which the varieties $\X(\A_i)$ are not isomorphic as subvarieties of $\P^2\times\P^8$ (Example~\ref{ex:Yu93}).

As an application, we use the Hirzebruch-Riemann-Roch
formula to describe the asymptotic behaviour of the Hilbert series
of the defining ideal $I(\A)$ of the variety $\X(\A)$: see 
Theorem~\ref{th:HS}.  In this way, the highest-degree terms of
the Hilbert polynomial and Hilbert series of $I(\A)$ are seen to
be reparameterizations of the characteristic polynomial, while the
lower-degree terms are not combinatorially determined: again,
the arrangements of Example~\ref{ex:Yu93} have $h(\Omega^1(\A_1),t)\neq h(\Omega^1(\A_2),t)$.  

In \cite{ST87}, Solomon and Terao express the characteristic polynomial formula for a central hyperplane arrangement \A\ in terms of a specialization of an alternating sum of Hilbert series of the modules of logarithmic forms, $\Omega^\bullet(\A)$.  This is necessarily somewhat delicate, in view of the remarks above.
However, if $\A$ is a free arrangement, the modules of logarithmic derivations form a free resolution of $I(\A)$. 
More generally, if $\A$ is a tame arrangement, one has an exact complex of logarithmic forms \eqref{eq:logcplx}, by \cite[Thm.~3.5]{CDFV09}.  In this case, an Euler characteristic argument, together with the results above, gives a geometric proof of Solomon and Terao's famous formula, Theorem~\ref{th:st}.  

%%%%%%%%%%%%%%%%%%%%%%%%%%%%%%%%%%%%%%%%%%%%%%%%%%%%%%%%%%%%%%%%%%%%%%%%%%%%%%%
\section{Background and notation}\label{sec:background}
%%%%%%%%%%%%%%%%%%%%%%%%%%%%%%%%%%%%%%%%%%%%%%%%%%%%%%%%%%%%%%%%%%%%%%%%%%%%%%%

Let $\A$ be a central arrangement in a complex vector space $V$ and set 
\[
m=\dim V,\quad n=\abs{\A}.
\]
For each $H\in\A$, we choose $f_H\in V^*$ with $H=\ker f_H$.
We further pick an ordering $\A=\{H_1,\dots,H_n\}$ and set $f_i=f_{H_i}$.

We refer to the book of Orlik and Terao~\cite{OTbook} for the notation and terminology of hyperplane arrangements not given here. 
We will use Fulton's book~\cite{FultonBook} as our reference for notation and basic facts about intersection theory.

%%%%%%%%%%%%%%%%%%%%%%%%%%%%%%%%%%%%%%%%%%%%%%%%%%%%%%%%%%%%%%%%%%%%%%%%%%%%%%%
\subsection{Combinatorics}
%%%%%%%%%%%%%%%%%%%%%%%%%%%%%%%%%%%%%%%%%%%%%%%%%%%%%%%%%%%%%%%%%%%%%%%%%%%%%%%

Let $L(\A)$ denote the intersection lattice of $\A$, ordered by reverse inclusion.
The {\em rank} of $X\in\A$ is, by definition, the codimension of $X$ in $V$, for which we wite $\rank(X)$.
By definition, the rank 
\[
\ell=\rank\A
\]
of $\A$ is $\rank(W)$, where $W=\bigcap_{H\in\A}H$ is the maximal element of $L(\A)$. 
If $W=\set{0}$, the arrangement is called {\em essential} (in which
case $\ell=m$.)
Then the characteristic polynomial of $\A$ is defined to be
\begin{equation}\label{eq:chpoly}
\chi_{\A}(t):=\sum_{X\in L(\A)}\mu(V,X)t^{\ell-\rank X},
\end{equation}
where $\mu$ denotes the M\"obius function of $L(\A)$.
(Note that our definition of $\chi_\A(t)$ is the conventional one for matroids; however, this differs for non-essential arrangements from the definition of \cite[Def.~2.52]{OTbook}.)

For each $H\in\A$, the deletion and restriction at $H$ are hyperplane arrangements in $V$ and $H$, respectively, defined by $\A':=\A\setminus \set{ H}$ and $\A'':=\set{H\cap H'\mid H'\in \A'}$. 
$H$ is called a {\em bridge} (or separator) if $\rank(\A')<\rank(\A)$.
This means that $H$ is not in the span of the hyperplanes of $\A'$, so $H$ is a bridge if and only if $\A'$ is not essential.

For each $H\in\A$, the characteristic polynomial satisfies the
``deletion-restriction'' recurrence relation:
\begin{equation}\label{eq:delres}
\chi(\A,t)=\chi(\A',t)-\chi(\A'',t);
\end{equation}
if $H$ is a bridge, this simplifies to $\chi(\A,t)=(t-1)\chi(\A',t)$.

A hyperplane arrangement is a matroid representation, and so it has a 
Tutte polynomial, which we denote by $T_\A(x,y)$. 
The characteristic polynomial is the univariate specialization 
\begin{equation}\label{eq:tutte}
T_\A(x,0)=(-1)^\ell\chi(\A,1-x).
\end{equation}
Denote by $t_{ij}$ the coefficient of $x^iy^j$ in $T_\A(x,y)$. 
Recall that $t_{ij}\ge0$ and $t_{00}=0$ (see, e.g., \cite[Thm.~6.2.13]{BrOx92}), as well as 
\begin{equation}\label{eq:Thomog}
\chihom(\A,-s,t-s)=\sum_{i=1}^\ell t_{i0}s^{\ell-i}t^i.
\end{equation}

%%%%%%%%%%%%%%%%%%%%%%%%%%%%%%%%%%%%%%%%%%%%%%%%%%%%%%%%%%%%%%%%%%%%%%%%%%%%%%%
\subsection{Logarithmic forms and derivations}\label{sec:log}
%%%%%%%%%%%%%%%%%%%%%%%%%%%%%%%%%%%%%%%%%%%%%%%%%%%%%%%%%%%%%%%%%%%%%%%%%%%%%%%

Let $R=\C[V]$, the coordinate ring of $V$.
The module of logarithmic $p$-derivations of $\A$ is defined by 
\[
D_p(\A)=\set{\theta\in\bigwedge^p\Der_\C(R)\mid\forall g_2,\dots, g_m\in R\colon\theta (f,g_2,\dots,g_m)\in (f)}.
\]
where $\Der_\C(R)$ is the module of $\C$-linear $p$-derivations over the ring $R$ which consist of $p$-linear skew-symmetric maps $\theta:R^p\rightarrow R$ which satisfy the Leibniz rule in every factor.
Dually, for $0\leq p\leq m$, the module of logarithmic $p$-forms is, by definition,
\[
\Omega^p(\A)=\set{\omega\in\Omega^p_{R/\C,f}\mid f\omega\in\Omega^p_{R/\C},fd\omega\in\Omega^{p+1}_{R/\C}},
\]
where $\Omega^p_{R/\C}$ is the module $\C$-linear K\"{a}hler differential $p$-forms.

Via the action of $\Der_\C(R)$ on $R$, $D(\A)$ inherits a grading that induces gradings of $D_p(\A)$ and $\Omega^p(\A)$.
In particular, the Euler derivation has degree $0$.

Solomon and Terao \cite{ST87} established a remarkable formula for the characteristic polynomial of an arrangement, expressed in terms of the Hilbert series of the modules of logarithmic derivations.
Let $h(-,t)$ denote Hilbert series (in the $R$-variables). 

Since our grading convention differs from that of Solomon and Terao by a shift of degree $p$, their formula becomes:

\begin{thm}[\cite{ST87}]\label{th:st}
The formal power series
\[
\Psi_\A(s,t)=\sum_{p=0}^m h(D_p(\A),t)t^p(s(1-t)-1)^p
\]
is, in fact, a Laurent polynomial in $\Z[s,t,t^{-1}]$, and its specialization $\Psi_\A(s,1)=(-1)^m\chi(\A,-s)$.
\end{thm}

For convenience, we define
\begin{equation}\label{eq:defP}
P_\A(x,y)=\sum_{p=0}^m h(D_p(\A),x)y^p.
\end{equation}
so that then $P_\A(t,st(1-t)-t)=\Psi_\A(s,t)$.

%%%%%%%%%%%%%%%%%%%%%%%%%%%%%%%%%%%%%%%%%%%%%%%%%%%%%%%%%%%%%%%%%%%%%%%%%%%%%%%
\subsection{Reducibility}\label{reducibility}
%%%%%%%%%%%%%%%%%%%%%%%%%%%%%%%%%%%%%%%%%%%%%%%%%%%%%%%%%%%%%%%%%%%%%%%%%%%%%%%

Recall that $\A$ in $V$ is called reducible if there is a non-trivial decomposition $V=V_1\oplus V_2$ and $\A=\A_1\cup\A_2$ for arrangements $\A_1$ and $\A_2$ in $V_1$ and $V_2$ respectively. 
In this case, we write $\A=\A_1\oplus\A_2$.
Otherwise, $\A$ is said to be irreducible.  
Note that $\A$ has a bridge if and only if $\A$ is reducible and $\A=\A_1\oplus\A_2$ in such a way that $\abs{\A_1}=1$.

A decomposition $\A=\A_1\oplus\A_2$ induces a decomposition
\begin{equation}\label{eqn:Dsplit}
D(\A)=(D(\A_1)\otimes_{\C}R_2) \oplus (R_1\otimes_{\C} D(\A_2)),
\end{equation}
where $R_i=\C[V_i]$. 
In particular, if $H=\ker x_1$ is a bridge and $\A'$ is the deletion, then
\begin{equation}\label{eqn:bridgedecomp}
D(\A)=R x_1\partial_{x_1}\oplus D(\A').
\end{equation}

%%%%%%%%%%%%%%%%%%%%%%%%%%%%%%%%%%%%%%%%%%%%%%%%%%%%%%%%%%%%%%%%%%%%%%%%%%%%%%%
\subsection{Critical sets}\label{ss:critical}
%%%%%%%%%%%%%%%%%%%%%%%%%%%%%%%%%%%%%%%%%%%%%%%%%%%%%%%%%%%%%%%%%%%%%%%%%%%%%%%

The variety of critical points of $\A$ is, by definition,
\[
\Sigma(\A):=\set{(x,\lambda)\in M\times\C^\A\mid \omega_{\lambda}(x)=0}.
\]
Let $C=C(\A)=\C[\C^\A]=\C[a_H\mid H\in\A]$, write $a_i=a_{H_i}$, and set $S=R\otimes_{\C}C$. 

\begin{thm}[{\cite[Thm.~2.9]{CDFV09}}]\label{th:closure}
For any central arrangement $\A$, we have $\OSigma(\A)=V(I(\A))$, where
\[
I(\A):=(\langle\theta,\omega_\ba\rangle\mid\theta\in D(\A)\otimes_R S),
\]
and
$\langle\cdot,\cdot\rangle$ denotes the contraction of a $1$-form along a logarithmic vector field.
\end{thm}

Since $\A$ is central, $I(\A)$ is bihomogeneous in the variables of $R$ and $C$ respectively. 
By \cite[Cor.~3.8]{CDFV09}, $I(\A)$ is radical if the arrangement $\A$ is tame. 
The variety $\OSigma(\A)$ is irreducible and, in general, singular.

\begin{exa}\label{ex:3lines}
For the arrangement $\A$ defined by $xy(x-y)$ in $\C^2$, we may take
$\theta_1=x\partial_x+y\partial_y$ and 
$\theta_2=x^2\partial_x+y^2\partial_y$ as a basis for the module of
derivations.  Then
\[
I(\A)=(a_1+a_2+a_3,x(a_1+a_3)+y(a_2+a_3)).
\]
\end{exa}

\begin{exa}\label{ex:boolean}
If $\A$ is the Boolean arrangement, then $I(\A)=(a_1,\ldots,a_n)$. Note that this is the irrelevant ideal of $C$, so in this case
$\X(\A)$ is empty.
\end{exa}

%%%%%%%%%%%%%%%%%%%%%%%%%%%%%%%%%%%%%%%%%%%%%%%%%%%%%%%%%%%%%%%%%%%%%%%%%%%%%%%
\section{A deletion-restriction formula}\label{sec:drform}
%%%%%%%%%%%%%%%%%%%%%%%%%%%%%%%%%%%%%%%%%%%%%%%%%%%%%%%%%%%%%%%%%%%%%%%%%%%%%%%

In this section, we will assume that $\A$ is an essential arrangement, so $\ell=m$.
We fix a hyperplane $H\in\A$ relative to which we define the deletion $\A'$, the restriction $\A''$, and the multirestriction $\A^H$. 

In order to compare the varieties $\OSigma$ for arrangements $\A$, $\A'$, and $\A''$, we shall introduce a diagram
\begin{equation}\label{eq:maps}
\xymat{
H\times\C^{\A''} & H\times\C^{\A'}\ar@{->>}[l]_-\sigma\ar@{^(->}[r]^\rho & V\times\C^{\A'}\ar@{^(->}[r]^-\iota & V\times \C^\A,
}
\end{equation}
where $\rho$ and $\iota$ are closed immersions, and $\sigma$ is a linear projection.

First, let $\rho\colon H\into V$ be the natural inclusion. 
Then define the linear projection $\sigma\colon\C^{\A^H}\onto\C^{\A''}$ as the $\C$-linear extension of the canonical surjection $\sigma\colon\A^H\onto\A''$ on coordinates.
Similarly, define the linear inclusion $\iota\colon\C^{\A'}\into\C^{\A}$ as the $\C$-linear extension of the  canonical inclusion $\iota\colon\A'\into\A$ on coordinates.
The corresponding maps of coordinate rings $\sigma^*\colon C(\A'')\into C(\A^H)$ and $\iota^*\colon C(\A)\onto C(\A')$ are given by
\[
\sigma^*(a_{H''})=\sum_{\sigma(H')=H''}a_{H'},\quad\iota^*(a_{H'})=a_{\iota(H')},
\]
respectively. 
Finally, abusing notation, write $\rho$ for $\rho\times\id$, $\sigma$ for $\id\times\sigma$, and $\iota$ for $\id\times\iota$.
This completes the definition of the diagram \eqref{eq:maps}.

\begin{thm}\label{th:delres}
For any arrangement $\A$, if $H$ is a bridge, then
\begin{equation}\label{5}
\OSigma(\A)=\OSigma(\A)\cap V(a_H)=\iota(\OSigma(\A'))\supset\iota\rho\sigma^{-1}(\OSigma(\A''));
\end{equation}
otherwise,
\begin{equation}\label{2}
\OSigma(\A)\cap V(a_H)=\iota(\OSigma(\A'))\cup\iota\rho\sigma^{-1}(\OSigma(\A'')),
\end{equation}
and this is generically a transversal intersection of smooth varieties.
\end{thm}

In terms of the the diagram of rings
\[
\xymat{
R/(f_H)\otimes C(\A'')\ar[r]^-{\sigma^*}&R/(f_H)\otimes C/(a_H)&\ar[l]_-{\rho^*}R\otimes C/(a_H)&\ar[l]_-{\iota^*}R\otimes C=S
}
\]
associated with \eqref{eq:maps}, the varieties in Theorem~\ref{th:delres} are given explicitly by
\begin{equation}\label{eq:defid}
\iota(\OSigma(\A'))=V((\iota^*)^{-1}(I(\A'))),\quad\iota\rho\sigma^{-1}(\OSigma(\A''))=V((\iota^*\circ\rho^*)^{-1}(\langle D(\A''),\omega_{\ba}^H\rangle)
\end{equation}
where $\omega_{\ba}^H=\rho^*\omega_{\bf{a'}}$. 

We first settle two special cases of Theorem~\ref{th:delres}.

\begin{lem}
If $H$ is a bridge, then \eqref{5} holds.
\end{lem}

\begin{proof}
By \eqref{eq:defid}, it suffices to verify that
\[
I(\A)+(a_H)=(\iota^*)^{-1}(I(\A'))\subset(\iota^*\circ\rho^*)^{-1}(\langle  D(\A''),\omega_{\ba}^{H}\rangle).
\] 
This follows immediately from \eqref{eqn:bridgedecomp}.
\end{proof}

\begin{lem}\label{10}
We have
\[
\OSigma(\A)\cap V(a_H)\cap D(f_H)=\iota\OSigma(\A')\cap D(f_H).
\]
Moreover, $\overline{\Sigma}(\A)$ and $V(a_H)$ have a generically smooth and transversal intersection if $H$ is not a bridge. 
\end{lem}

\begin{proof}
We compute
\begin{align*}
(I(\A)+(a_H))_{f_H}
&=(a_H,\left<\theta,\omega_{\ba}\right>\mid \theta\in D(\A))_{f_H}\\
&=(a_H,\left<\theta,\omega_{\ba}\right>\mid \theta\in D(\A)_{f_H})\\
&=(a_H,\left<\theta,\omega_{\ba'}\right>\mid \theta\in D(\A)_{f_H})\\
&=(a_H,\left<\theta,\omega_{\ba'}\right>\mid \theta\in D(\A')_{f_H})\\
&=(a_H,\left<\theta,\omega_{\ba'}\right>\mid \theta\in D(\A'))_{f_H}\\
&=(\iota^*)^{-1}(I(\A'))_{f_H}.
\end{align*}
Using \eqref{eq:defid} this proves the first claim.

The last claim follows from \cite[Prop.~2.5]{CDFV09}, which states that the projection $M\times\C^\A\to M$ turns $\Sigma(\A)$ into a vector bundle over $M$ of rank $\abs{\A}-\ell$. 
Since $H$ is not a bridge, $\rank(\A')=\ell$ as well. 
By the same result, $\Sigma(\A')$ is then a vector bundle over $M$ of rank equal to  $\abs{A'}-\ell=n-1-\ell$.
Thus, by the first claim, $V(a_H)$ must intersect each fiber of $\Sigma(\A)$ over $M$ transversally.
The second claim follows.
\end{proof}

\begin{proof}[Proof of Theorem~\ref{th:delres}]
We shall assume that $H=H_1$.
For brevity, denote
\[
W=\OSigma(\A),\ W'=\iota\OSigma(\A'),\ W''=\iota\rho\sigma^{-1}\OSigma(\A''),\ K=V(a_1),\ Z=W\cap K.
\]
For any subset $X\subset V$, we denote by a lower index $X$ the intersection with $X\times\C^n$, and by $V_X(-)$ the zero set in $X\times\C^n$ of a collection of equations.
Using the inclusion $D(\A)\subset D(\A')$ and restriction $D(\A)\to D(\A'')$ one readily verifies that $Z\supseteq W'\cup W''$.
By Lemma~\ref{10}, $W_{V\setminus H}\cap K_{V\setminus H}=W'_{V\setminus H}$ intersects as claimed.

Let $X^\circ=X\setminus\bigcup(\A\setminus\A_X)$ for any flat $X$. 
To prove \eqref{2}, then, it is enough to show for each $X$ that
\begin{equation}\label{4}
Z_{X^\circ}=W_{X^\circ}\cap K_{X^\circ}\subseteq W'_{X^\circ}\cup W''_{X^\circ}.
\end{equation}
The previous paragraph shows this holds for $X\not\subseteq H$, so we assume $X\subseteq H$.
By \cite[Prop.~2.4]{CDFV09}, $W$, $W'$, $W''$ are irreducible of dimensions $n$, $n-1$, $n-1$ respectively.
So the irreducible components of $W\cap K$ have dimension $n-1$ or $n$.
Since $W\cap K$ is decomposed into finitely many constructible sets $W_{X^\circ}\cap K_{X^\circ}$, it is then not necessary to prove \eqref{4} in case $\dim W_{X^\circ}\cap K_{X^\circ}<n-1$, and in particular not in case $\dim W_{X^\circ}<n-1$.

Let $d=\dim X$.
We may assume that $\A_X=\{H_1,\dots,H_k\}$ and, since $\A$ is essential, that
\begin{equation}\label{1}
X\cap H_{k+1}\cap\dots\cap H_{k+d}=\{0\}.
\end{equation}
In particular, any $f_j$ is a linear combination of $f_1,\dots,f_{k+d}$.
Applying the localization technique in \cite[p.~13]{CDFV09}, we may then replace $\omega_\ba$ by $\omega_{\ba'}$ where $\ba'=(a_1,\dots,a_{k+d},0,\dots,0)$, by a suitable coordinate change.
Note that coordinates $a_i$, $i=1,\dots,k$, are indeed unchanged over ${X^\circ}$:
They are changed only by a multiple of $f_i$ which is zero on $X$.
Using \eqref{1}, we may then choose a coordinate system on $V$ such that $x_{\ell-d+i}=f_{k+i}$, $i=1,\dots,d$, and $f_j=f_j(x_1,\dots,x_{\ell-d})$, $j=1,\dots,k$.
Set $T=V(x_{\ell-d+1},\dots,x_\ell)$ such that $p+T$ is transversal to $X$ at $p\in{X^\circ}$. 
Then, locally at $p$, $D(\A)$ is generated by $\partial_{\ell-d+1},\dots,\partial_\ell$ and $D(\A_X^{p+T})$, where $\A_X^{p+T}=\{H'\cap(p+T)\mid H'\in\A_X\}$.
Using \cite[Prop.~2.8]{CDFV09}, it follows that
\begin{equation}\label{3}
W_{X^\circ}=V_{X^\circ}(a_1+\dots+a_k,a_{k+1},\dots,a_{k+d})
\end{equation}
if $\A_X$ is irreducible.

In general, the sum could split into several such sums according to an irreducible decomposition of $\A_X$. 
But this case is irrelevant because then $\dim W_{X^\circ}<n-1$.
Intersecting $W_{X^\circ}$ with $K_{X^\circ}$ means adding $a_1$ to \eqref{3}.
Then the dimension drops to $n-2$ unless $k=1$, so we can assume $k=1$ which means $X=H$ and hence $d=\ell-1$.
But then $W_{X^\circ}=W'_{X^\circ}=W''_{X^\circ}=V(a_1,\dots,a_\ell)$ and \eqref{4} holds trivially.

To prove the statement on generic transversality, we are reduced to the case $X=H$ as before, but we have to find the equations of $W_{V^\circ_H}$ where $V^\circ_H=V\setminus\bigcup(\A\setminus\{H\})$.
To this end, denote by $\ba$ and $\ba'$ the coordinates before and after the coordinate change from \cite[p.~13]{CDFV09} applied above.
With the above choice of coordinates and $x_1=f_1$, 
\[
a'_i=a_i+x_i\sum_{j>\ell}c_{ij}\frac{a_j}{f_j},\quad i\le\ell,\quad
a'_j=a_j,\quad j>\ell.
\]
Since $D(\A)$ is generated by $x_1\partial_1,\partial_2,\dots,\partial_\ell$ on $V^\circ_H$, it follows that
\[
W_{V^\circ_H}=V_{V^\circ_H}\left(a_1+x_1r,\frac{a'_2}{x_2},\dots,\frac{a'_\ell}{x_\ell}\right)=V_{V^\circ_H}(a_1+x_1r,a'_2,\dots,a'_\ell)
\]
where
\[
r=\sum_{j>\ell}c_{1,j}\frac{a_j}{f_j}=\sum_{j>\ell}c_{1,j}\frac{a'_j}{f_j}.
\]
If $H$ is not a bridge, $c_{1,j}\ne0$ for some $j$ and hence $r\ne0$ on
\[
V_{V^\circ_H}(x_1,a'_1,\dots,a'_\ell)=V_{H^\circ}(a_1,a'_2,\dots,a'_\ell)=Z_{H^\circ}.
\]
Thus, generically along $Z_{H^\circ}$, $W_{V^\circ_H}$ is smooth and intersects $K_{V^\circ_H}=V_{V^\circ_H}(a_1)$ transversally.
\end{proof}

%%%%%%%%%%%%%%%%%%%%%%%%%%%%%%%%%%%%%%%%%%%%%%%%%%%%%%%%%%%%%%%%%%%%%%%%%%%%%%%
\section{An intersection ring formula}
%%%%%%%%%%%%%%%%%%%%%%%%%%%%%%%%%%%%%%%%%%%%%%%%%%%%%%%%%%%%%%%%%%%%%%%%%%%%%%%

In this section, we prove Theorem~\ref{thm:intcycle} using a deletion-restriction argument based on Theorem~\ref{th:delres}.
We begin with two terms in the equality of Theorem~\ref{thm:intcycle} that need to be verified separately.  

\begin{lem}\label{lem:monic}
Then the coefficient of $h^{\ell}$ in $[\X(\A)]$ is zero. 
Moreover, if $\A$ is not Boolean, then the coefficient of $k^{\ell}$ equals $1$. 
\end{lem}

\begin{proof}
First, $\X(\A)$ is contained in $\P V\times\P K$, where $K$ is the hyperplane in $\C^n$ given by $\sum_{i=1}^n a_i=0$, by \cite[Prop.~2.6]{CDFV09}. 
Therefore $k^{n-1}\cdot[\X(\A)]=[\X(\A)\cap\P V\times\set{\lambda}]=0$, for any
$\lambda\not\in\P\C^\A-\P K$, which gives the first claim.

Next, by \cite[Prop.~4.1]{OT95}, $\Sigma(\A)$ is a vector bundle of rank $n-\ell$ over the complement $M\subseteq V$. 
The torus $\C^*\times\C^*$ acts compatibly (see \cite[Prop.~2.5]{CDFV09}). 
Since $\A$ is not Boolean, we have $n>\ell$. 
For any $x\in M$, $\X(\A)\cap(\set{x}\times\P^{n-1})$ is then rationally equivalent to $\set{x}\times\P^{n-\ell-1}$.
That is, $h^{\ell-1}\cdot [\X(\A)]=h^{\ell-1}k^{\ell}$. 
The second claim follows.
\end{proof}

We continue to assume that $\A$ is essential.
Both to justify this hypothesis and for the following proof for essential $\A$, we will need the pullback of cycles along a rational map coming from a linear projection.
Due to the lack of an obvious reference, we give the construction.

Let $\pi\colon V'\onto V''$ be a linear projection of $\C$-vector spaces, $V$ another $\C$-vector space, and set $d=\dim V''$.
Let $Y=\P V\times \P(\ker \pi)$, and $U=X\setminus Y$.
\[
\xymat{
Y\ar@{^(->}[r]^-\alpha & X=\P V\times\P V'\ar@{-->}[r]^-\pi & \P V\times\P V''=Z\\
&U\ar@{^(->}[u]^-\beta\ar[ru]_-{\pi_U}
}
\]
where $\pi=\id\times\pi$ is a rational map with domain $U$.
\begin{lem}\label{lem:ratpullback}
Using the notation above, let
\begin{equation}\label{eq:ratpullbackinv}
(\pi^*)^{-1}=(\pi_U^*)^{-1}\circ\beta^*.
\end{equation}
Then the following sequence is exact:
\begin{equation}\label{eq:rightexact}
\xymat{
A^{\bullet-d}(Y)\ar[r]^-{\alpha_*} & A^\bullet(X)\ar[r]^-{(\pi^*)^{-1}} & A^\bullet(Z)\ar[r] & 0.
}
\end{equation}
Restricted to codimension $p<d$, $(\pi^*)^{-1}$ is an isomorphism with inverse
\begin{equation}\label{eq:ratpullback}
(\pi^*)^p=((\beta^*)^p)^{-1}\circ(\pi_U^*)^p\colon A^p(Z)\to A^p(X).
\end{equation}
These maps constitute an additive map that we shall denote by $\pi^*$.
\end{lem}

\begin{proof}
Since $\pi_U$ is a vector bundle, the flat pullback $\pi_U^*$ is an isomorphism (see \cite[Thm.~3.3]{FultonBook}).
Set $Y=X\setminus U$ and note that $\codim Y=d$ by hypothesis on $\pi$.
Then the flat pullback $\beta^*$ is surjective (by \cite[Prop.~1.8]{FultonBook}) and
\[
(\beta^*)^p\colon A^p(X)\onto A^p(U)
\]
is an isomorphism for $p<\codim Y$.  The claim follows.
\end{proof}

\begin{rmk}\label{rem:essential}
Assume that $\A$ is not essential.
Let $\A^e$ denote the essential arrangement obtained as the image of $\A$ under the linear projection $\pi\colon V\onto V/W$, where we recall $W=\bigcap_{H\in\A}H$. 
Applying Lemma~\ref{lem:ratpullback} to the corresponding rational map
\[
\xymat{
\pi=\pi\times\id\colon\P V\times\P\C^{\A}\ar@{-->}[r] & \P(V/W)\times\P\C^{\A^e},
}
\]
we obtain $(\pi^*)^{-1}[\X(\A)]=[\X(\A^e)]$ where $(\pi^*)^{-1}$ is defined by \eqref{eq:ratpullbackinv}. 
Since $[\X(\A)]\in A^\ell$ where $\ell=\codim W$, the map $((\pi^*)^{-1})^\ell$ is not an isomorphism.  
However, by \eqref{eq:rightexact}, its kernel is generated by $\alpha_*[W]=h^\ell$. 
But by Lemma~\ref{lem:monic}, the coefficient of $h^\ell$ in $[\X(\A)]$ is zero, so it suffices to prove Theorem~\ref{thm:intcycle} for essential arrangements.
\end{rmk}

We begin with the base case of an induction argument.

\begin{rmk}\label{rem:boolean}
Let $\A$ be the Boolean arrangement. 
Then both sides of the formula of Theorem \ref{thm:intcycle} are zero: $\X(\A)$ is the empty variety, defined by the irrelevant ideal in the second factor (Example \ref{ex:boolean}). 
On the other hand, $\chihom(\A,-h,k-h)$ equals $k^\ell$, which is zero in the Chow ring $A^\bullet$. 
\end{rmk}

For the induction step, we return to the setup of \S\ref{sec:drform}.
By further abuse of notation, we let $\sigma$, $\rho$, and $\iota$ denote the projectivization of the maps of \eqref{eq:maps}:
\begin{equation}\label{eq:projmaps}
\xymat{
\P H\times\P\C^{\A''}& \P H\times\P\C^{\A'}\ar@{-->}[l]_-\sigma\ar@{^(->}[r]^-\rho & \P V\times\P\C^{\A'}\ar@{^(->}[r]^-\iota & \P V\times\P\C^{\A}.
}
\end{equation}
Since $[\X(\A'')]\in A^{\ell-1}(\P H\times\P\C^{\A''})$, $\sigma^*[\X(\A'')]$ is defined by \eqref{eq:ratpullback} and, by definition,
\begin{equation}\label{eq:qpullback}
\sigma^*[\X(\A'')]=[\sigma^{-1}\X(\A'')].
\end{equation}

The geometric formula of Theorem~\ref{th:delres} now leads to the following in $A^{\bullet}$:

\begin{prp}\label{prop:inductclass}
If $H$ is a bridge, then 
\begin{equation}\label{eq:bridgecase}
[\X(\A)]=k\cdot[\X(\A')];
\end{equation}
otherwise, 
\begin{equation}\label{eq:recurrence}
k\cdot[\X(\A)]=k\cdot[\X(\A')]+
\begin{cases}
hk^\ell & \text{if $\A''$ is Boolean,}\\
hk\cdot\sigma^*[\X(\A'')] & \text{otherwise.}
\end{cases}
\end{equation}
\end{prp}

\begin{proof}
If $H$ is a bridge, then $[\X(\A)]=[i\X(\A')]$ by \eqref{5}. 
Since $\iota$ is a linear inclusion, it is proper and has degree $1$, so $[\iota\X(\A)]=\iota_*[\X(\A')]$.
By the projection formula, $\iota_*(x\cdot\iota^*(k))=k\cdot\iota_*(x)$ for
$x\in A^\bullet(\P\C^{\A'})$ and, since $k=[K]$ and $K=\iota(\P\C^{\A'})$, $\iota^*(h)=1$ by definition. 
This proves \eqref{eq:bridgecase}.

If $H$ is not a bridge, by \eqref{2}, we have
\[
k\cdot[\X(\A)]=[\iota\X(\A')]+[\iota\rho\sigma^{-1}\X(\A'')],
\]
using \cite[Rem.~8.2]{FultonBook}.
If $\A''$ is Boolean, then $\iota\rho\sigma^{-1}(\OSigma(\A''))$ is a product of $H$ with a codimension-$\ell$ linear subspace of $\C^\A$ (Remark~\ref{rem:boolean}). 
Otherwise, $\sigma^{-1}\X(\A'')$ is nonempty, and the proof of \eqref{eq:recurrence} uses \eqref{eq:qpullback} and the same arguments as in the bridge case.
\end{proof}

We are ready to prove our main result.

\begin{proof}[Proof of Theorem~\ref{thm:intcycle}]
We argue by induction on $n=\abs{\A}$, the base case being trivial by Remark~\ref{rem:boolean}. 
For indeterminates $s$ and $t$, the recurrence \eqref{eq:delres} becomes
\begin{equation}\label{eq:homorecurrence}
\chihom(\A,-s,t-s)=\begin{cases}
t\chihom(\A',-s,t-s)&\text{if $H$ is a bridge,}\\
\chihom(\A',-s,t-s)+s\chihom(\A'',-s,t-s)&\text{otherwise.}
\end{cases}
\end{equation}
If $H$ is a bridge, then $[\X(\A)]=k\cdot \chihom(\A',-h,k-h)$, by induction together with Proposition~\ref{prop:inductclass} and Remark~\ref{rem:essential}.

If $H$ is not a bridge,
\[
k\cdot [\X(\A)]=k\cdot \chihom(\A',-h,k-h)+hk\cdot\chihom(\A'',-h,k-h),
\]
by induction and Proposition~\ref{prop:inductclass}.
Both sides of the expression have degree $\ell+1$.  We may assume 
$\A$ is not Boolean (Remark~\ref{rem:boolean}), in which case $n\geq
\ell+1$.  If $n>\ell+1$, we can conclude that 
\[
[\X(\A)]= \chihom(\A',-h,k-h)+h\cdot\chihom(\A'',-h,k-h).
\]
If $n=\ell+1$, then $k^{\ell+1}=0$ in $A^\bullet$.  In this case,
the coefficient of $k^\ell$ on the left is $1$ by Lemma~\ref{lem:monic}, and the same on the right since $T_\A(x,0)$ is monic, using \eqref{eq:Thomog}.
\end{proof}
\begin{rmk}
Let $L$ denote a line in $\P\C^\A$.  Then by Theorem~\ref{thm:intcycle},
\begin{eqnarray*}
[\X(\A)]\cdot[L] &=& \chi^+(-h,k-h)k^{n-2}\\
&=& t_{10}h^{\ell-1}k^{n-1}.
\end{eqnarray*}
The coefficient $t_{10}$ equals $\beta(\A)=\abs{\chi(\P M)}$, the well-known beta invariant of $\A$ (see, e.g., \cite[Prop.~6.2.12]{BrOx92}.)

This is to say, by Bezout's Theorem, that for generic choices of
$\lambda\in\C^n$ for which $\sum_{i=1}^n\lambda_i=0$, the 
critical set of the master function $\Phi_\lambda$ in $V$ equals
$\beta(\A)$ points.  The main result of \cite{OT95} is a refined
version of this statement: they show, additionally, that the
critical points are isolated and nondegenerate.  This calculation is also closely related to \cite[Thm.~5(3)]{CHKS06}, where the authors
count critical points (with multiplicities) for global normal
crossings divisors.
\end{rmk}

%%%%%%%%%%%%%%%%%%%%%%%%%%%%%%%%%%%%%%%%%%%%%%%%%%%%%%%%%%%%%%%%%%%%%%%%%%%%%%%
\section{Application to Hilbert series}
%%%%%%%%%%%%%%%%%%%%%%%%%%%%%%%%%%%%%%%%%%%%%%%%%%%%%%%%%%%%%%%%%%%%%%%%%%%%%%%

%%%%%%%%%%%%%%%%%%%%%%%%%%%%%%%%%%%%%%%%%%%%%%%%%%%%%%%%%%%%%%%%%%%%%%%%%%%%%%%
\subsection{Chern classes}
%%%%%%%%%%%%%%%%%%%%%%%%%%%%%%%%%%%%%%%%%%%%%%%%%%%%%%%%%%%%%%%%%%%%%%%%%%%%%%%

Recall that one may define Chern classes for any coherent sheaf $\F$ on a nonsingular variety $Y$: in this case, $\F$ has a finite resolution by vector bundles, and the Chern classes of $\F$ are defined formally using the resolution via the Whitney sum formula.  If the support of $\F$ has codimension $\ell$, then $c_p(\F)=0$ for $0\leq p<\ell$, by \cite[Ex.~15.3.6]{FultonBook}.
If $Z$ is a codimension-$\ell$ subvariety of $Y$, we have moreover that $c_\ell(\O_Z)=(-1)^{\ell-1}(\ell-1)![Z]$ in $A(Y)$, by \cite[Ex.~15.3.6]{FultonBook}.

Applying this to our situation gives the following, using Theorem~\ref{thm:intcycle}.

\begin{prp}\label{prop:chern_l}
For any arrangement $\A$, we have $c_p(\O_{\X(\A)})=0$ for $0\leq p<\ell$ and
\[
c_\ell(\O_{\X(\A)})=-(\ell-1)!\chihom(\A,h,h-k).
\]
\end{prp}

In terms of the exponential Chern character, by \cite[Ex.~15.1.2(c),~Ex.~15.2.16(a)]{FultonBook},
\begin{equation}\label{eq:chOX}
\ch(\O_{\X(\A)})=\chihom(\A,-h,k-h)+O(\set{h,k}^\ell),
\end{equation}
where $O(\set{h,k}^\ell)$ denotes a polynomial in $A^\bullet\otimes_\Z\Q$ whose monomials are all of total degree strictly greater than $\ell$.

%%%%%%%%%%%%%%%%%%%%%%%%%%%%%%%%%%%%%%%%%%%%%%%%%%%%%%%%%%%%%%%%%%%%%%%%%%%%%%%
\subsection{The Hilbert polynomial}
%%%%%%%%%%%%%%%%%%%%%%%%%%%%%%%%%%%%%%%%%%%%%%%%%%%%%%%%%%%%%%%%%%%%%%%%%%%%%%%

For a bigraded $S$-module $M$, let $p_M(p,q)$ denote its Hilbert polynomial: i.e., $p_M(a,b)={\dim_\C}M_{a,b}$ for integers $a,b\gg0$. 
We refer to the classic paper of van der Waerden~\cite{vdW28} for properties of Hilbert series and Hilbert polynomials of bigraded modules. 
In particular, the (total) degree of $p_M(a,b)$ equals $\dim_S M-2$.

In this section, fix an arrangement $\A$ and set $I=I(\A)$ and $\X=\X(\A)$. 
It turns out that the asymptotic behaviour of the Hilbert polynomial of $S/I$ is combinatorially determined.  

\begin{prp}\label{prop:hpoly}
If $\A$ is a rank-$\ell$ arrangement of $n\geq2$ hyperplanes, then
\begin{equation}\label{eq:hilbertpoly}
p_{S/I}(p,q)=\frac1{(n-2)!}\sum_{i=1}^\ell t_{i0}{n-2\choose i-1}p^{i-1}q^{n-1-i}\;+\;\Omega(\set{p,q}^{n-2}),
\end{equation}
where $\Omega(\set{p,q}^{n-2})$ denotes a polynomial in $p$ and $q$ of
total degree strictly less than $n-2$.
\end{prp}

\begin{proof}
By the Hirzebruch-Riemann-Roch formula (see also \cite[Exc.~III.5.2]{Har77}),
\[
p_{S/I}(a,b)=\int\ch(\O_\X(a,b))\td(T_{\P V\times\P^{n-1}}),
\]
for all nonnegative integers $a,b$. 
Then
\begin{equation}\label{eq:hp2}
p_{S/I}(p,q)=[h^{\ell-1}k^{n-1}]\ch(\O_\X)e^{ph+qk}\td(T_{\P V\times\P^{n-1}}),
\end{equation}
where $[h^ik^j](-)$ denotes the coefficient of $h^ik^j$ in an 
element of the ring $A^\bullet[p,q]$.  This is a polynomial of degree $\dim\X=n-2$ in $p$ and $q$; in this proof, we will refer to the gradings in $A^\bullet$ and variables $p,q$ as the $hk$-degree and $pq$-degrees, respectively.%\query{$(p,q)$-degree?  G: too bulky?}

Consider the product expansion of \eqref{eq:hp2}. 
Terms in the middle factor have matching $pq$- and $hk$-degrees.
In order to obtain a term in the product of 
$hk$-degree $n+\ell-2$ and $pq$-degree $n-2$, then, only the least nonzero terms of $\ch(\O_\X)$ and $\td(T_{\P V\times\P^{n-1}})$ may appear, by \eqref{eq:chOX}. 
As the Todd polynomial has constant term $1$ and using \eqref{eq:Thomog}, the highest degree term of $p_{S/I}(p,q)$ can be written
\begin{align*}
[h^{\ell-1}k^{n-1}]\chihom(\A,-h,k-h)e^{ph+qk}
&=\sum_{i=1}^\ell t_{i0}[h^{\ell-1-(\ell-i)}k^{n-1-i}]\frac{(ph+qk)^{n-2}}{(n-2)!}\\
&=\frac{1}{(n-2)!}\sum_{i=1}^\ell t_{i0}{n-2\choose i-1}p^{i-1}q^{n-1-i},
\end{align*}
and the claim follows.
\end{proof}

\begin{rmk}
If $\A$ is a Boolean arrangement of rank $\ell$, then $S/I\cong R$.  So $p_{S/I}(p,q)=0$, and the Hilbert series is $h(S/I;t,u)=(1-t)^{-\ell}$.
\end{rmk}

%%%%%%%%%%%%%%%%%%%%%%%%%%%%%%%%%%%%%%%%%%%%%%%%%%%%%%%%%%%%%%%%%%%%%%%%%%%%%%%
\subsection{The Hilbert series}
%%%%%%%%%%%%%%%%%%%%%%%%%%%%%%%%%%%%%%%%%%%%%%%%%%%%%%%%%%%%%%%%%%%%%%%%%%%%%%%

In this section, we assume that $\A$ is not Boolean, to avoid the degenerate special case.
The result from the previous section may also be expressed in terms of the Hilbert series.  
From \cite[Thm.~7]{vdW28}, the Hilbert series of $S/I$ can be written as
\[
h(S/I;t,u)=\sum_{i=0}^\ell\frac{g_i(t,u)}{(1-t)^i(1-u)^{n-i}}
\]
for some polynomials $g_i(t,u)$, $0\leq i\leq \ell$.  By means of a 
partial fractions expansion, the series $h(S/I;t,u)$ may also be written as
\begin{equation}\label{eq:genericHS}
h(S/I;t,u)=\sum_{\substack{i,j\geq0\\i+j\leq n}}
\frac{c_{ij}}{(1-t)^i(1-u)^j}
\end{equation}
where the coefficients $\set{c_{ij}}$ are integers for $i,j\geq1$,
$c_{0j}\in\Z[t]$ for $j\geq1$, $c_{i0}\in\Z[u]$ for $i\geq1$, and
$c_{00}\in\Z[t,u]$.% \query{reference?} GD: don't have one

The terms in this expansion of highest pole order are combinatorially
determined:
\begin{thm}\label{th:HS}
For any arrangement $\A$ of $n$ hyperplanes,
\begin{equation}\label{eq:HS}
h(S/I;t,u)=\sum_{i=1}^{\ell}
\frac{t_{i0}}{(1-t)^i(1-u)^{n-i}}+\Omega(\{(1-t)^{-1},(1-u)^{-1}\}^{n}).
\end{equation}
\end{thm}

\begin{proof}
Via the binomial expansion, \eqref{eq:genericHS} becomes
\begin{equation}\label{eq:HSexpanded}
h(S/I;t,u)=\sum_{\substack{i,j,p,q\geq0\\i+j\leq n}}
c_{ij}\binom{p+i-1}{p}\binom{q+j-1}{q} t^pu^q.
\end{equation}
For any $k$, the highest degree term of $\binom{p+k}{p}$ (as a polynomial in $p$) equals $1/k!p^k$.  Then, for each $i$, $1\leq i\leq n-1$, the coefficient of $pq$-degree $(i-1,n-i-1)$ in \eqref{eq:HSexpanded}
equals $c_{i,n-i}/((i-1)!(n-i-1)!)$.  However, this coefficient also equals $t_{i0}\binom{n-2}{i-1}/(n-2)!$, by \eqref{eq:hilbertpoly}, so $c_{i-1,n-i-1}=t_{i0}$.  The argument is completed by noting that
$t_{00}=0$ and $t_{i0}=0$ for $i>\ell$.
\end{proof}

\begin{cor}\label{cor:tutte}
The formal power series 
\[
(1-t+st(1-t))^{n}h(S/I;t,t-st(1-t))
\]
is a polynomial in $s$ and $t$.  Its evaluation at $t=1$ is $(-1)^\ell\chi(\A,-s)$.
\end{cor}

\begin{proof}
By Theorem~\ref{th:HS}, 
\[
(1-u)^n h(S/I;t,u)=T_\A(\frac{1-u}{1-t},0)+(1-u)^nQ(t,u),
\]
where $Q(t,u)$ is some formal power series with (total) pole order
strictly less than $n$.

Now apply the change of variables
$t\mapsto t$, $u\mapsto t-st(1-t)$.  Since $1-u\mapsto (1-t)(1+st)$,
we see $(1-t)^{n-1} Q(t,t-st(1-t))$ is, in fact, a polynomial.
Then the first claim follows by writing out the substitution:
\begin{multline*}
(1+t-st(1-t))^nh(S/I,t,t-st(1-t)) = \\ T_\A(1+st,0)+(1-t)^n(1+st)^nQ(t,t-st(1-t)).
\end{multline*}
Since the second summand is a polynomial divisible by $1-t$, the 
second claim follows by setting $t=1$ and using \eqref{eq:tutte}.
\end{proof}

\begin{exa}
Let $\A$ be arrangement of Example \ref{ex:3lines}.  Here,  $\chi(t)=(t-1)(t-2)$ and $T_{\A}(x,y)=x^2+x+y$: by 
Theorem~\ref{thm:intcycle} we find $[\X(\A)]=kh+k^2$. By direct computation, we have
\begin{align*}
\quad p_{S/I}(p,q)&=p+q+1,\\
h(S/I,t,u)&=\frac{1}{(1-t)(1-u)^2}+\frac{1}{(1-t)^2(1-u)}-\frac{1}{(1-t)(1-u)},
\end{align*}
where the leading parts are predicted by Proposition~\ref{prop:hpoly} and Corollary~\ref{cor:tutte}.
\end{exa}

\begin{exa}\label{ex:Yu93}
In \cite[Ex.~8.7]{Zi89}, Ziegler found a pair of rank-$3$ arrangements $\A_1$ and $\A_2$ with isomorphic intersection lattices, for which $\Omega^1(\A_i)$ have different Hilbert series, for $i=1,2$.  Viewed in
$\P^2$, these are arrangements of $9$ lines with six triple points.
We will use the realizations of Yuzvinsky~\cite[Ex.~2.2]{Yu93}. 
Schenck and Toh\v{a}neanu considered the same example as \cite[Ex.~1.4]{ShTo09}, and they observed that the triple points
of the one arrangement lie on a conic, while those of the other do not.
The Tutte polynomial of either specializes to
\[
 T_{\A_i}(x,0)=x^3+6x^2+15x.
\]
Computation with \cite{M2} shows that the respective Hilbert series of $S/I(\A_i)$ are:
 \begin{align*}
h_1 &= \frac{1-6t^5u+4t^6u+t^6u^2}{(1-t)^3(1-u)^8}\text{~and}\\
h_2 &= \frac{1-t^4u-3t^5u+t^6u+t^6u^2+t^7u}{(1-t)^3(1-u)^8}.
\end{align*}
Then
\[
h_i=\frac{15}{(1-t)(1-u)^8}+\frac{6}{(1-t)^2(1-u)^7}+\frac{1}{(1-t)^3(1-u)^6}+
\Omega_i
\]
for $i=1,2$; however, the tails differ: $h_1-h_2=\Omega_1-\Omega_2=t^4u/(1-u)^8$.

Since these arrangements are tame, the defining ideals are radical by \cite[Cor.~3.8]{CDFV09}. 
Since their respective Hilbert series differ, $\X(\A_1)\not\cong\X(\A_2)$ as subvarieties of $\P^2\times\P^8$.
\end{exa}

%%%%%%%%%%%%%%%%%%%%%%%%%%%%%%%%%%%%%%%%%%%%%%%%%%%%%%%%%%%%%%%%%%%%%%%%%%%%%%%
\subsection{Solomon and Terao's formula}
%%%%%%%%%%%%%%%%%%%%%%%%%%%%%%%%%%%%%%%%%%%%%%%%%%%%%%%%%%%%%%%%%%%%%%%%%%%%%%%

If $\A$ is a tame arrangement, we can obtain another formula for $\ch(\O_{\X(\A)})$ and compare with the identity \eqref{eq:chOX}.
To begin, recall the following (using $[\cdot,\cdot]$ to denote a shift of bidegrees):

\begin{thm}[{\cite[Thm~3.5]{CDFV09}}]
For any tame arrangement $\A$, the complex
\begin{gather}
\xymat@C-0.7em{
0\ar[r] &
\Omega^0_{S/\kk}(\A)[0,-\ell] \ar[r]^-{\omega_\ba} &
\Omega^1_{S/\kk}(\A)[0,1-\ell] \ar[r]^-{\omega_\ba} &
\cdots}\nonumber\\
\xymat@C-0.7em{
&\cdots \ar[r]^-{\omega_\ba} &
\Omega^{\ell-1}_{S/\kk}[0,-1]\ar[r]^{\omega_\ba} &
\Omega^\ell_{S/\kk}[0,0]\ar[r] & (S/I)[n-\ell,0]\ar[r] & 0
}\label{eq:logcplx}
\end{gather}
is an exact complex of bigraded $S$-modules.
\end{thm}

We may replace \eqref{eq:logcplx} by the following, using the identity $D_p(\A)\cong \Omega^{\ell-p}[\ell-n]$:

\begin{equation}\label{eq:Dcplx}
\xymat@C-1em{
0\ar[r]&D^{S/C}_\ell(\A)[0,-\ell]\ar[r] & \cdots\ar[r] &
D_1^{S/C}(\A)[0,-1]\ar[r] & D_0^{S/C}(\A)\ar[r] & S/I\ar[r] &0,
}
\end{equation}
where the differential is contraction along $\omega_\ba$.  Noting that
$D_p^{S/C}(\A)=D_p(\A)\otimes_R C$ for $0\leq p\leq \ell$, we obtain:

\begin{cor}\label{cor:Dcplx}
If $\A$ is a tame arrangement, then the bigraded Hilbert series of $S/I$ is given by
\[
h(S/I,t,u)=P_\A(t,-u)/(1-u)^{n},
\]
where $P_\A(t,u)$ is defined in \eqref{eq:defP}.
\end{cor}

\begin{rmk}
If $\A$ is a tame arrangement, or any arrangement for which the complex \eqref{eq:logcplx} is exact, then Solomon and Terao's formula for $\A$ (Theorem~\ref{th:st}) is a consequence of Corollary~\ref{cor:Dcplx} together with Corollary~\ref{cor:tutte}.  
It is not known if \eqref{eq:logcplx} is exact for all arrangements.
\end{rmk}

%%%%%%%%%%%%%%%%%%%%%%%%%%%%%%%%%%%%%%%%%%%%%%%%%%%%%%%%%%%%%%%%%%%%%%%%%%%%%%%%

\bibliographystyle{amsalpha}
\providecommand{\bysame}{\leavevmode\hbox to3em{\hrulefill}\thinspace}
\providecommand{\MR}{\relax\ifhmode\unskip\space\fi MR }
% \MRhref is called by the amsart/book/proc definition of \MR.
\providecommand{\MRhref}[2]{%
  \href{http://www.ams.org/mathscinet-getitem?mr=#1}{#2}
}
\providecommand{\href}[2]{#2}

%%%%%%%%%%%%%%%%%%%%%%%%%%%%%%%%%%%%%%%%%%%%%%%%%%%%%%%%%%%%%%%%%%%%%%%%%%%%%%%%
\end{document}